\documentclass{amsart}

\usepackage{mathtools}
\usepackage{amsfonts}
\usepackage{amssymb}
\usepackage{amsthm}
\usepackage{mathrsfs}
\usepackage[all]{xy}
\usepackage{amsmath}
\usepackage{color}
\usepackage{graphicx}
\usepackage{stmaryrd}
\usepackage{xspace}
\usepackage{xfrac}
\usepackage{etex,etoolbox}

\mathtoolsset{showonlyrefs} 

\theoremstyle{plain}
\newtheorem{proposition}[subsubsection]{Proposition}
\newtheorem{lemma}[subsubsection]{Lemma}
\newtheorem{theorem}[subsubsection]{Theorem}
\newtheorem{corollary}[subsubsection]{Corollary}

\theoremstyle{definition}

\theoremstyle{remark}

\newtheorem{remark}[subsubsection]{Remark}

\newcommand{\PF}{Poisson\xspace} 

\newcommand{\Aff}{\operatorname{Aff}} 
\newcommand{\Hor}{\operatorname{Hor}} 
\newcommand{\Aut}{\operatorname{Aut}} 

\newcommand{\Stab}[2]{\operatorname{Stab}_{#2}({#1})} 

\newcommand{\NN}{\mathbb{N}} 
\newcommand{\PI}{\mathbb{N}_0} 
\newcommand{\I}{\mathbb{Z}} 

\newcommand{\supp}{\operatorname{supp}}

\newcommand{\sgr}{\operatorname{sgr}} 
\newcommand{\gr}{\operatorname{gr}} 
\newcommand{\pf}[2]{{#1}_*{#2}} 
\newcommand{\mm}[1]{m_1 \left ( {#1} \right )} 
\newcommand{\convpow}[2]{{#1}^{\ast #2}}

\newcommand{\EV}{\mathbb{E}} 

\newcommand{\Gmu}{(G,  \mu)} 
\newcommand{\BHluc}{H^\infty_{luc}} 

\newcommand{\Td}[1]{\mathcal{T}_{{#1}}}
\newcommand{\Tdp}{\mathcal{T}}

\newcommand{\flr}[1]{\left \lfloor {#1}  \right \rfloor} 

\newcommand{\imesi}[1]{\mu_{{#1}}}

\newcommand{\pathspace}{G^\NN}
\newcommand{\prodmeasure}{\mathbf{P}^\mu}
\newcommand{\pathmeasure}{\mathbb{P}^\mu}

\newcommand{\Vmma}{V_{--} \rtimes \langle \alpha \rangle}
\newcommand{\Vmm}{V_{--} }
\newcommand{\Vm}{V_{-} }

\newcommand{\bnd}{\mathbf{bnd}}

\usepackage{amsmath,amssymb,graphicx,url} 
\usepackage{blindtext}
\usepackage{enumerate}

\begin{document}

\title[Random walks on products of trees and certain t.d.l.c. groups]{Random walks on products of trees and certain totally disconnected, locally compact groups}
\author{John J. Harrison}
\address{School of Mathematical and Physical Sciences, 
The University of Newcastle,
Callaghan NSW 2308, Australia}
\email{John.Harrison@newcastle.edu.au}
\date{\today}

\begin{abstract}
	The \PF boundary of a finite direct product of affine automorphism groups of homogeneous trees is considered. The \PF boundary is shown to be a product of ends of trees with a hitting measure for spread-out, aperiodic measures of finite first moment whose closed support generates a subgroup which is not \emph{fully exceptional}. The Poisson boundary of a semi-direct product $\Vmma$ for any automorphism $\alpha$ and tidy compact open subgroup $V$ in a locally compact, totally disconnected group $G$ is also shown to be the space of ends of the tree with the hitting measure under similar assumptions. Necessary and sufficient conditions for boundary triviality are given in both cases. The method of proof is largely an extension of the prior work of Cartwright, Kaimanovich and Woess on affine automorphism groups of homogeneous trees.
\end{abstract}
\maketitle
\tableofcontents

\section{Introduction and preliminaries}
Cartwright, Kaimanovich and Woess studied random walks on the \emph{affine automorphism group of a homogeneous tree}, which we denote by $\Aff \Tdp$, in \cite{cartwright94}. They established a law of large numbers, a central limit theorem and described the \PF boundary for measures which have finite first moment with respect to a gauge function, $| \cdot |_\Tdp$ and whose support generates a closed subgroup which are \emph{non-exceptional}. 

In the preliminary section, we recall conventions, definitions and results from the literature relating to totally disconnected, locally compact groups, random walks, the \PF boundary, gauges, gauge functions and the affine automorphism group of a homogeneous tree. Particular attention is given to the results of the Cartwright, Kaimanovich and Woess paper.

We extend their work in two directions. In the first part, we consider closed subgroups in a finite direct product of affine automorphism groups of homogeneous trees, $P = \prod_{i=1}^k \Aff \Td{i}$. We define \emph{partially exceptional} and \emph{fully exceptional} closed subgroups of $P$, and explore the relationships between these properties, the modular function, transience of the random walk and the scale function. To ensure transience, we restrict our analysis to random walks for probability measures on $P$ whose support generates a subgroup which is not fully exceptional.

When the probability measure on $P$ is spread-out, aperiodic, and has finite first moment with respect to a subadditive gauge function, $|\cdot|_P$, we show that the \PF boundary, $(B, \nu)$, is the direct product of the space ends of each tree with a hitting measure. We give necessary and sufficient conditions for boundary triviality. We conclude by showing that a transitive action of $\Gamma$ on the ends is equivalent to $\Gamma$ being compactly generated simplifies the proof that $(B, \nu)$ is the \PF boundary.

In the last part, we consider groups of the form $\Vmma$, where $V$ is a subgroup, tidy for an automorphism $\alpha$, of a totally disconnected, locally compact group. According to Baumgartner and Willis \cite{baumgartner04}, these groups act naturally on an infinite homogeneous tree with degree dependent on the scale of $\alpha$. We describe the \PF boundary for random walks on these groups in the case where the probability measure generates a \emph{non-extraordinary} subgroup, is spread-out and has finite first moment with respect to a gauge function closely related to $| \cdot |_\Tdp$.

\section{Preliminaries}

\subsection{Totally disconnected, locally compact groups}
Many definitions and results in this section are due to Willis \cite{willis94,willis01}. Let $G$ be any totally disconnected, locally compact (\emph{t.d.l.c}) group and let $\alpha$ be a continuous automorphism of $G$.   For each compact open subgroup $V$ of $G$, let 
\[ V_+ = \bigcap_{k \geq 0 } \alpha^k (V) \ \textrm{and} \ V_- = \bigcap_{k \geq 0} \alpha^{-k} (V). \]
The subgroup $ V_0 = V_+ \cap V_-$ is always closed. If $V = V_+ V_-$, then $V$ is \emph{tidy above} with respect to $\alpha$ and if the subgroup
\[ V_{++} = \bigcup_{k \geq 0} \alpha^k (V_+) \]
is closed, then $V$ is \emph{tidy below} with respect to $\alpha$. The subgroup 
\[ V_{--} = \bigcup_{k \geq 0} \alpha^{-k} (V_-) \]
is closed if and only if $V_{++}$ is closed. If $V$ is both tidy above and tidy below, then it is \emph{tidy for $\alpha$}.  Tidy compact open subgroups for $\alpha$ can be constructed from arbitrary compact open subgroups with a \emph{tidying procedure}, as detailed in Willis \cite{willis01scale}. The \emph{scale of $\alpha$} is the positive integer
\[ s(\alpha) := \min \{ [ \alpha (U) : U \cap \alpha(U) ] \mid U \ \textrm{is a compact open subgroup of} \ G \}. \]
Van Dantzig's Theorem ensures the existence of at least one compact open subgroup $U$. The scale is finite because $[ \alpha (U) : U \cap \alpha(U) ]$ is finite by the compactness of $U$. Any compact open subgroup $V$ satisfying $ s(\alpha) = [ \alpha (V) : V \cap \alpha(V) ]$ is \emph{minimizing for $\alpha$}\index{minimizing subgroup}. It was shown, by Willis \cite{willis04}, that every compact open subgroup which is tidy for $\alpha$ is minimizing for $\alpha$. Hence, the scale of $\alpha$ is given by
\[ s(\alpha) = [ \alpha (V) : V \cap \alpha(V) ] = [ \alpha (V_+) : V_+ ], \]
where $V$ is any compact open subgroup tidy for $\alpha$. 

Suppose that $g$ is a group element in $G$. The \emph{scale of $g$} is given by
\[ s(g) = s(\alpha_g), \]
where $\alpha_g$ is conjugation by $g$. The scale function $s$ from $G$ to the natural numbers satisfies \[ \Delta(g) = \frac{s(g)}{s(g^{-1})},\] where $\Delta$ is the modular function on $\Aut (G)$.  A subset of group elements $E$ is \emph{uniscalar} if the scale function is identically $1$ on $E$. It follows that uniscalar groups are unimodular.

\subsection{Random walks and the Poisson boundary}
Suppose that $G$ is a second countable locally compact group with identity $e$ and that $\mu$ is a probability measure on $G$. Any such pair $(G, \mu)$ is called a \emph{random walk}.

The \emph{space of trajectories} is the set $G^\NN$ with the product $\sigma$-algebra, where $G^\NN$ is infinite Cartesian product of countably many copies of $G$. An element $\omega$ in $G^\NN$ is a \emph{trajectory} or \emph{path}. 

We denote by $\pathmeasure$ the pushforward of  $\mu^\NN$  with respect to the map $S$ on $\pathspace$,  $\pathmeasure = \mu^\NN \circ S^{-1}$, given by
\[ S(\omega_1, \omega_2, \omega_3, \dots, \omega_k, \dots) = (\omega_1, \omega_1 \omega_2, \omega_1 \omega_2 \omega_3, \dots,  \omega_1 \dots \omega_k , \dots ). \] 
 The measure  $\pathmeasure$ is called the \emph{path measure}, and the pair $(\pathspace, \pathmeasure)$ is the \emph{path space}. We identify the random walk $\Gmu$ with a discrete time-homogeneous Markov chain $\{R_i\}_{i \in \PI }$, called the \emph{right random walk}, where each random variable $R_n$ is the projection from the path space,
\[ R_n(\omega) = \omega_n   \]
The group $G$ acts diagonally on elements of the path space. This action extends to one on probability measures on $G^\NN$, namely, if $m$ is any probability measure on $G^\NN$, then 
\[ g \cdot m(E)  = m(g^{-1} E) \]
for each measurable set $E$ and $g$ in $G$. 

There are many equivalent definitions of the \PF boundary of a random walk. See e.g. Erschler \cite{erschler10}, Furstenberg \cite{furstenberg63,furstenberg71}, or Kaimanovich and Vershik \cite{kaimanovich83}.

If $X$ is a topological space, then the pair $(X, \cdot)$ is a \emph{$G$-space} if $G$ acts on $X$ and the map $(g,f) \mapsto g \cdot f$ from $G \times X$ to $X$ is continuous with respect to the product topology on $G \times X$. 

A measure $\nu$ on a $G$-space $B$ is said to be \emph{$\mu$-stationary}  if $\mu \ast \nu = \nu$. A $G$-space $B$ equipped with a $\mu$-stationary measure $\nu$ is called a \emph{$\Gmu$-space}.

Let $B$ be a second countable $G$-space and $(B, \nu)$ be a $\Gmu$-space.  Then, $(B, \nu)$ is said to be a \emph{$\mu$-boundary} if there exists a random variable $\bnd$ from $\pathspace$ to $B$,  called the \emph{boundary map}, such that $R_n(\omega) \cdot \nu$ converges in the weak* topology to a point measure $\delta_{\bnd(\omega)}$ for $\pathmeasure$-almost every $\omega$ in $\pathspace$.

A function $f$ in $L^\infty (G)$ is a bounded $\mu$-harmonic function if it satisfies the convolution identity $f = f \ast \mu$. The set of all bounded $\mu$-harmonic functions with pointwise addition, complex conjugation and the multiplication
\[ \lim_{n \rightarrow \infty}  \left ( (f\cdot g) \ast \convpow{\mu}{n} \right )  (x)\]
is a $C^*$ algebra, which we denote by $\BHluc\Gmu$. If $B$ is a compact $\Gmu$-space with a $\mu$-stationary measure $\nu$, then 
\[ P_\nu (\varphi)(g) = \int_B \varphi(g b) \, d \nu(b).  \] 
is an element of  $\BHluc\Gmu$. The map $ P_\nu$ is called the \emph{Poisson transform}. A second countable $\Gmu$-space $(B, \nu)$ is a $\mu$-boundary if and only if Poisson transformation $P_\nu$ is a $\ast$-homomorphism.

A continuous map $\gamma$ from a $G$-space $B$ to a $G$-space $B'$ is \emph{equivariant} if $\gamma(g \cdot b) = g \cdot \gamma (b)$ for all $b$ in $B$ and $g$ in $G$. If $(B, \nu)$ and $(\bar{B}, \bar{\nu})$ are $\mu$-boundaries of the random walk $(G, \nu)$ then $(\bar{B}, \bar{\nu})$ is an \emph{equivariant image of $(B, \nu)$} if there exists an equivariant map $\gamma$ from $B$ to $B'$, such that the pushforward measure $\pf{\gamma}{\nu}$ is equal to $\bar{\nu}$.   

Given any random walk $\Gmu$,  there is a $\mu$-boundary $(\Pi_\mu, \nu)$, such that the Poisson map $P_\nu$ is an isometric *-isomorphism and every other $\mu$-boundary $(B, \eta)$ is an equivariant image of $(\Pi_\mu, \nu)$. 

A \emph{gauge} is an increasing sequence $\mathcal{A}$ of measurable sets $\mathcal{A}_j$ which exhaust $G$. A \emph{gauge function} is a non-negative integer-valued function $\delta$ for which there exists a non-negative constant $K$, such that 
\[ \delta(g h) \leq \delta (g) + \delta (h) + K \]
for all $g$ and $h$ in $G$. A gauge function is subadditive if 
\[ \delta(g h) \leq \delta (g) + \delta (h) \]
for all $g$ and $h$ in $G$.  Let $\mathcal{A} = \{ \mathcal{A} \}_{i=1}^\infty$ be a gauge.  Then, $\mathcal{A}$ is \emph{subadditive} if the \emph{gauge map}
\[ | \gamma |_{\mathcal{A}} = \min \left \{k \in \NN: \gamma \in \mathcal{A}_k \right \} \]
is a subadditive gauge function. If $\delta$ is a gauge function, then the sequence $\mathcal{A}^\delta = \{ \mathcal{A}^\delta_i \}_{i=1}^\infty$ given by
\[ \mathcal{A}^\delta_j = \{ g \in G : \delta(g) \leq j \} \] 
is a gauge. If $\delta$ is a subadditive gauge function, then $\mathcal{A}^\delta_j$ is a subadditive gauge. We say that the gauge $\mathcal{A}$ is \emph{$C$-temperate}, or just \emph{temperate}, if $\lambda_G (\mathcal{A}_j) \leq e^{Cj}$ for all natural numbers $k$ and some positive real number $C$.  A sequence of gauges $\mathcal{A}^{(j)}$ is \emph{uniformly temperate} if there is a positive real number $C$, such that $\mathcal{A}^{(j)}$ is $C$-temperate for each natural number $j$.  See Kaimanovich \cite{kaimanovich91,kaimanovich00} for more details about gauges and gauge functions. 

We make use of a formulation of Kaimanovich's ray criterion for topological groups, given in Cartwright, Kaimanovich and Woess \cite{cartwright94} without proof. Kaimanovich has indicated to the author in personal correspondence that the relevant preprint has not yet been published.

\begin{theorem}[Kaimanovich's ray criterion for topological groups \cite{cartwright94, kaimanovich94}]
	\label{thm:kaimanovichapproxthmtop}
	Let $G$ be a second countable Hausdorff topological group. Let $\mu$ be an aperiodic, spread-out probability measure on $G$ with finite first moment with respect to a subadditive gauge $\mathcal{A}$ on $G$. Suppose that $(B, \nu)$ is a $\mu$-boundary with boundary map $\bnd$. If for $\nu$ almost every point $b$ in $B$ there is a uniformly temperate sequence of gauges $\mathcal{G}^n = \mathcal{G}^n(b)$, such that 
	\[ \frac{1}{n} | \omega_n |_{\mathcal{G}^n (\bnd (\omega))} \rightarrow 0  \]
	for almost every path $\omega = (\omega_1, \omega_2, \dots) $, then $(B, \nu)$ is the \PF boundary of the pair $\Gmu$.  
\end{theorem}
\begin{corollary}
	\label{cor:kaimanovichapproxthmtop}
	Let $G$, $\mu$, $ B$, $\nu$ and $\gamma$ be as in Theorem \ref{thm:kaimanovichapproxthmtop}. Suppose that $G$ is generated by a compact set $K$. Let $d$ be the word length metric. If there exists a sequence of measurable \emph{approximation maps} $\Pi_m \colon B \rightarrow G$, such that
	\[ \frac{1}{n} d(\Pi_m (\bnd (\omega)), \omega_1 \omega_2 \dots \omega_n ) \rightarrow 0  \]
	for almost every path $\omega = (\omega_1, \omega_2, \dots) $, then $(B, \nu)$ is the \PF boundary of the pair $\Gmu$.  
\end{corollary}
\subsection{The affine automorphism group of a homogeneous tree}
\label{sec:affhmtree}
Suppose that $d$ is a natural number greater than one. Let $\Tdp$ be the \emph{homogeneous tree of valency $d$}, the unique connected graph with no cycles for which every vertex has $d$ neighbours. We suppose that $d$ is at least $3$, and use the term \emph{tree} to mean a homogeneous tree with some valency.  Let $o$ be a distinguished vertex of $\Tdp$.

A \emph{path} in a tree is a sequence of successive neighbouring vertices without backtracking. Paths may be \emph{finite}, \emph{singly infinite} or \emph{doubly infinite}. The \emph{length} of a finite path $v = \{v_i \}_{i=j}^{k}$ is denoted by $|v|$ and equal to $k-j$. The \emph{length} of an infinite path is $+\infty$. Let $v$ and $w$ be vertices of the tree $\Tdp$. The \emph{geodesic segment between $v$ and $w$} is the unique finite path, $\overline{vw}$, from $v$ to $w$. The \emph{distance} between $v$ and $w$, $d(v,w)$ is the number of edges between the two vertices, $\left | \overline{vw} \right | - 1$. The distance from the distinguished vertex $o$ to a vertex $v$ in $\Tdp$ is $|v| = d(o,v)$.

A \emph{geodesic ray} in a tree is a singly infinite sequence of successive neighbours without backtracking. Two geodesic rays are \emph{equivalent} if their intersection is infinite. Each equivalence class of rays is an \emph{end}. The set of ends of a tree $\Tdp$ is denoted by $\partial \Tdp$ and the disjoint union of the tree with its ends by $\Tdp \cup \partial \Tdp$. If $\xi$ is in $\Tdp \cup \partial \Tdp$ and $v$ is in $\Tdp$, then there is a unique geodesic ray, $\overline{v \xi}$ starting at $v$ which represents $\xi$.  Let $\omega$ be a distinguished end of the tree $\Tdp$.  The set of all ends except $\omega$ is denoted by $\partial^* \Tdp$.

Let $u$ and $v$ be elements of $\Tdp \cup \partial \Tdp$. Denote by $u \curlywedge v$ the last common element on the geodesics $\overline{ou}$ and $\overline{ow}$, unless $u = v$ and they are both in $\partial{\Tdp}$, in which case set $u \curlywedge v = u$. Let 
\[ \theta(u,v) = \begin{cases} q^{-|u \curlywedge v|}  & \textrm{if} \ u \neq v, \\ 0 & \textrm{if} \ u = v. \end{cases} \]
\begin{lemma}
	The map $\theta$ is an ultrametric on $\Tdp \cup \partial \Tdp$. With the topology from $\theta$, $\Tdp \cup \partial \Tdp$ is a totally disconnected, compact space. The set $\Tdp $ is an open, dense and discrete subset in $\Tdp \cup \partial \Tdp$ and $\partial \Tdp$ is a closed, compact subset of $\Tdp \cup \partial \Tdp$. 
\end{lemma}
\begin{proof}
	The ultrametric inequality follows from the identity 
	\[ |u \curlywedge v| \geq \min \{ |u \curlywedge w|, |v \curlywedge w|  \}  \]
	for all $u$, $v$ and $w$ in $\Tdp \cup \partial \Tdp$. Every ultrametric space is totally disconnected. 
	
	If $u$ is a point in $\partial \Tdp$, then, for every natural number $k$,
	\begin{align*}
	B_{q^{-k}} \cap T &= \{ v \in T, \theta(u,v) \leq q^{-k} \} \\
	&= \{ v \in T, |u \curlywedge v| \geq k \}
	\end{align*}
	is always non-empty. Hence, $T$ is dense in $\Tdp \cup \partial \Tdp$.	
	
	Suppose that $v = \{v_i\}_{i=1}^\infty$ is a sequence in $\Tdp \cup \partial \Tdp$. If the sequence $\{|v_i|\}_{i=1}^\infty$ is bounded, then $v$  visits only finitely many points, so it has a convergent subsequence. If $\{|v_i|\}_{i=1}^\infty$ is unbounded, then there is an end $w$, such that for every natural number $k$, the open ball of radius $q^{-k}$ around $w$, $U_k = B_{q^{-k}}(w)$, contains infinitely many points from $v$. Choose a sequence of natural numbers $\{ n_i \}_{i=1}^\infty$, such that $v_{n_k}$ is in $U_k$. Then, $\{v_{n_i}\}_{i=1}^\infty$ is a subsequence of $v$ convergent to $w$. So $\Tdp \cup \partial \Tdp$ is sequentially compact, hence compact.

	Suppose that $u$ is a point in $T$. Let $k$ be any natural number greater than $|u|$. Then, 
	\[ B_{q^{-k}}(u) = \{ v \in \Tdp \cup \partial \Tdp : |u \curlywedge v| \geq k \} \cup \{ u \}.  \]
	Suppose that $v$ is an element of  $\Tdp \cup \partial \Tdp$ distinct from $u$. If $|u \curlywedge v| = 0$, then $v$ is not in $B_{q^{-k}}(u)$. 	If $|u \curlywedge v| \neq 0$, then there exists an end $w$ in $\partial \Tdp$, such that $u$ is in $\overline{o w}$ and either $u \curlywedge v = u \curlywedge w = u$ or $u \curlywedge v = v \curlywedge w$.	In the first case, 
	\[ k > | u \curlywedge v | = | u | \]
	and in the second case,
	\[ k > | u \curlywedge w | = | u \curlywedge v |. \]
	So $\{ v \in \Tdp \cup \partial \Tdp : |u \curlywedge v| \geq k \}$ is empty and $ B_{q^{-k}}(u)$ is just the singleton containing $u$.  Since $u$ was an arbitrary vertex, $T$ is discrete and open. Since $\partial \Tdp$ is the compliment of an open set, it is a closed, compact subset  of $\Tdp \cup \partial \Tdp$.
\end{proof}
An \emph{automorphism} $\varphi$ of the tree $\Tdp$ is a permutation of the vertex set of $\Tdp$, such that
\[ d(\varphi(v), \varphi(w)) = d(v,w) \]
for all vertices $v$ and $w$ in $\Tdp$. The set of all automorphisms of $\Tdp$ forms a group under function composition, called the \emph{automorphism group of $\Tdp$} and denoted by $\Aut \Tdp$. 

Tits \cite{tits70} showed that every automorphism of a tree is either \emph{elliptic}, fixing a vertex or inverting an edge, or \emph{hyperbolic}, acting as a translation along a doubly infinite path. Every automorphism of $\Tdp$ extends naturally to a homeomorphism of $\Tdp \cup \partial \Tdp$.

The automorphism group of $\Tdp$ is a totally disconnected, locally compact group with the \emph{compact-open topology}. The open neighbourhoods of this topology are unions of sets of the form
\[ \mathcal{U}(\alpha, \mathcal{F}) = \left \{ \beta \in \Aut \Tdp : \alpha v = \beta v \ \textrm{for all} \ v \ \textrm{in} \ \mathcal{F} \right \}, \] 
where $\mathcal{F}$ is any finite set of vertices and $\alpha$ is any automorphism of $\Tdp$. The compact open topology is also known as the \emph{topology of pointwise convergence}, because a sequence of automorphisms $\alpha_n$ converges to $\alpha$ if and only if for every finite set of vertices $F$, $\alpha_n(v)$ and $\alpha(v)$ eventually agree for all $v$ in $F$. As the degree of each vertex in $\Tdp$ is finite, the \emph{stabilizer of a vertex $v$} in $\Tdp$, 
\[ \operatorname{stab}_{\Aut \Tdp}(v) = \{ \beta \in \Aut \Tdp, \beta{v} = v \} \]
is a profinite group because it is isomorphic to an iterated wreath product of finite groups. All profinite groups are compact. If $\mathcal{F}$ is any finite set of vertices and $\alpha$ is an automorphism, then
\[ \mathcal{U}(\alpha, \mathcal{F}) = \bigcap_{v \in \mathcal{F}} \operatorname{stab}_{\Aut \Tdp}(\alpha v), \]
so each set $\mathcal{U}(\alpha, \mathcal{F})$ is compact. It follows that $\Aut \Tdp$ is locally compact. It is totally disconnected because there are arbitrarily small compact, open, hence closed, neighbourhoods of every point.  The topology is Hausdorff, because if $\alpha$ and $\beta$ are distinct automorphisms in $\Aut \Tdp$, then there is a vertex $v$, such that $\alpha v \neq \beta v$, so $\mathcal{U}(\alpha, \{v\}) $ and $\mathcal{U}(\beta, \{v\}) $ are neighbourhoods of $\alpha$ and $\beta$ respectively, but
\[ \mathcal{U}(\alpha, \{v\}) \cap \mathcal{U}(\beta, \{v\}) = \varnothing. \]
Let $v = \{v_i \}_{i=1}^\infty$ be a breadth-first sequence containing every vertex in $\Tdp$, i.e. such that $|v_i| \leq |v_j|$ whenever $i$ is less than $j$. Let $\mathcal{P}_k$ be the power set of the first $k$ vertices in $v$. Then, there is a natural number $k$, such that the set
\[ U_k = \left \{ \mathcal{U}(\alpha, \mathcal{F}) : \alpha \in \Aut \Tdp, \mathcal{F} \in \mathcal{P}_k \right \} \] 
contains $\mathcal{U}(\alpha, \mathcal{F})$ for any $\alpha$ in $\Aut \Tdp$ and finite set of vertices $\mathcal{F}$. It follows that $\Aut \Tdp$ is $\sigma$-compact and second countable. Let $\lambda_{\Aut \Tdp}$ be the right Haar measure on $\Aut \Tdp$, normalised so that vertex stabilizers have measure $1$. Since $\Aut \Tdp$ is not compact, $\lambda_{\Aut \Tdp}$ is an infinite measure.

Suppose that $\omega$ is a distinguished end of $\Tdp$. Let the \emph{Busemann function}, $h$ be the map from the vertices of $\Tdp$ to the natural numbers given by
\[ h(v) = d(v,c) - d(o,c),  \]
for each vertex $v$, where $c$ is the first common vertex on the geodesic rays $\overline{x \omega}$ and $\overline{o \omega}$. The \emph{horocyclic map} $\phi$ from $\Aut \Tdp$ to $\I$ given by 
\[ \phi(\alpha) = h(\alpha o) = h(\alpha v) - h(v) \]
for any vertex $v$ in $\Tdp$, is a well defined homomorphism, because automorphisms of $\Tdp$ are distance preserving under $d$, hence
\[ h(u) - h(v) = h(\alpha u) - h(\alpha v) \]
for all automorphisms $\alpha$ in $\Aut \Tdp$ and vertices $u$ and $v$ in $\Tdp$.
The \emph{horocyclic group}, denoted by $\Hor \Tdp$, is the kernel of $\phi$; the closed subgroup of $\Aut \Tdp$ which preserves each set of \emph{horocycles}, 
\[ H_m = \{ v \in \Tdp : h(v) = m  \}. \]
An element of $\Hor \Tdp$ is a \emph{horocyclic} automorphism. Every horocyclic automorphism $\alpha$ fixes the common ancestor of $v$ and $\alpha v$ for every vertex $v$ in $\Tdp$.

The \emph{affine group of the homogeneous tree} $\Tdp$ is the closed subgroup of all automorphisms of $\Tdp$ that fix the distinguished end $\omega$, with the subspace topology, and is denoted by $\Aff \Tdp$. The subspace topology on $\Aff \Tdp$ is second countable, locally compact,  totally disconnected and $\Aff \Tdp$ is $\sigma$-compact. 

 The affine group of $\Tdp$ is an internal semi-direct product, isomorphic to $\I \ltimes \Hor {\Tdp}$. To see this, fix an element $\sigma$ in $\Aff \Tdp$, such that $\sigma(o)$ is a child of $o$ and $\phi(\sigma) = 1$. The intersection of $\langle \sigma \rangle \cong \I$ and $\Hor \Tdp$ contains only the identity, and every element $\gamma$ in $\Aff \Tdp$ can be written as
\[ \gamma = \left ( \gamma \sigma^{-\phi(\gamma)} \right ) \sigma^{\phi(\gamma)} \]
where $\gamma \sigma^{-\phi(\gamma)}$ is in $\Hor \Tdp$ and $\sigma^{\phi(\gamma)}$ is in $\langle \sigma \rangle \cong \I$.

Given an element $\gamma$ in $\Aff \Tdp$, let $| \gamma |_{\Tdp} = | \gamma o |$. The map $| \cdot |_{\Tdp}$ is a subadditive gauge function, because
\begin{align*}
| \alpha \beta |_{\Tdp} &= d(o, \alpha \beta o ) \\
&\leq d(o, \alpha o ) + d(\alpha o, \alpha \beta o ) \\
&= d(o, \alpha o ) + d(o, \beta o ) \\
&= | \alpha |_{\Tdp} + | \beta |_{\Tdp}
\end{align*} 
for all automorphisms $\alpha$ and $\beta$ in $\Aff \Tdp$. The relation
\[ | \gamma |_{\Tdp} = d(o, \gamma o) = d(\gamma^{-1} o, o) = | \gamma^{-1} |_{\Tdp}  \]
is also satisfied for all $\gamma$ in $\Aff \Tdp$. Other properties of this gauge function are given in Lemma 4 of Cartwright, Kaimanovich and Woess \cite{cartwright94}.

\begin{lemma}
	Let $v$ be a vertex in $\Aff \Tdp$, Then
	\[ |v| = d(o,v) \geq | h(v) |. \]
\end{lemma}	
\begin{proof}
	Let $c$ is the first common vertex on the geodesic rays $\overline{v \omega}$ and $\overline{o \omega}$, so that
	\[  h(v) = d(x,c) - d(o,c). \]
	If $v$ is in $\overline{o \omega}$, then $c = v$, so $h(x) = -d(o,x)$. If $v$ is a child of $o$, then $c = o$, so $h(x,o)$. Suppose that $v$ is neither in $\overline{o \omega}$ nor a child of $o$, then
	\[ d(o,x) = d(o, c) + d(x,c). \]
	If $d(o, c) > d(x,c)$, then $d(o,x) > d(o, c) - d(x,c) = |h(x)|$. Similarly, if $d(x, c) > d(o,c)$, then $d(o,x) > d(x, c) - d(o,c) = |h(x)|$.
\end{proof}

If $\mu$ is a probability measure on $\Aff \Tdp$ with finite first moment with respect to $|\cdot|_{\Tdp}$, then pushforward measure, $\pf{\phi}{\mu}$, has finite first moment with respect to the ordinary absolute value on the integers as 
\begin{align*}
\mm{\pf{\phi}{\mu}} &= \int_{\I} |z| \, d \pf{\phi}{\mu} (z) \\
&= \int_{\Aff \Tdp} |\phi(\gamma)| \, d \mu (\gamma) \\
&= \int_{\Aff \Tdp} |h(\gamma o)| \, d  \mu (\gamma)\\
&\leq \int_{\Aff \Tdp} | \gamma o| \, d  \mu (\gamma)\\
&= \int_{\Aff \Tdp} |\gamma|_{\Tdp} \, d  \mu (\gamma).
\end{align*}	

A sequence of vertices $\{ v_i \}_{i=1}^\infty$ in $\Tdp$ is \emph{regular} if there is an end $\xi$ in $\partial \Tdp$ and a non-negative real number $a$, called the \emph{rate of escape}, such that
\[ a = \lim_{n \rightarrow \infty} \frac{1}{n} d (v_n, \xi_{\flr{a n}}) \]
where $v_n$ is the $n$th vertex on the geodesic $\overline{o \xi}$ and $\flr{\cdot}$ is the floor function. Notice that if $a$ is non-zero, then $v_n$ converges to $\xi$, so $\xi$ is unique.

\begin{lemma}[Cartwright, Kaimanovich and Woess \cite{cartwright94}, Lemma 3]
	\label{lem:regularity}
	A sequence of vertices $\{ v_i \}_{i=1}^\infty$ in $\Tdp$ is regular with rate of escape $a$ if and only if 
	\begin{enumerate}[(i)]
		\item $\lim_{n \rightarrow \infty} \frac{1}{n} d (v_n, v_{n+1}), $ and
		\item $\lim_{n \rightarrow \infty} \frac{1}{n} |v_n| = a$.  
	\end{enumerate}
\end{lemma}

\begin{proposition}[Cartwright, Kaimanovich and Woess \cite{cartwright94}, Proposition 1]
	\label{prop:regularity}
	A sequence of vertices $v = \{ v_i \}_{i=1}^\infty$ in $\Tdp$ is regular if and only if 
	\begin{enumerate}[(i)]
		\item $\lim_{n \rightarrow \infty} \frac{1}{n} d (v_n, v_{n+1}), $ and
		\item the limit $a_h = \lim_{n \rightarrow \infty} \frac{1}{n} h(v_n)$ exists.  
	\end{enumerate}
	In this case the rate of escape of $v$ is $|a_h|$.
\end{proposition}
Following Cartwright, Kaimanovich and Woess \cite{cartwright94}, we call a closed subgroup of $\Aff \Tdp$ with distinguished end $\omega$ \emph{exceptional} if it is contained within $\Hor \Tdp$ or there is an end, apart from $\omega$, which is fixed by every element. 
\begin{proposition}
	\label{prop:exceptionalimpliesuniscalar}
	Let $\Gamma$ be an exceptional closed subgroup in $\Aff \Tdp$ with distinguished end $\omega$. Then, $\Gamma$ is uniscalar.
\end{proposition}
\begin{proof}
	Suppose that $\Gamma$ is contained within $\Hor \Tdp$. Let $\gamma$ be in $\Gamma$. Let $\alpha$ be conjugation by $\gamma$ Then, $\alpha$ fixes the element $o \wedge \gamma o$.  Let
	\[ U = \operatorname{stab}_{\Gamma}(o \wedge \gamma o). \]
	Then, $U$ is compact and open and $\alpha(U) = U$, so
	\[ s(\gamma) = [ \alpha(U) : U \cap \alpha(U) ] = 1.  \]

	Suppose instead $\Gamma$ fixes an element $\xi$ in $\partial^* \Tdp$.  Then $\gamma$ acts as a translation along the doubly infinite path $\overline{\xi \omega}$.  Let $v$ be a vertex in $\overline{\xi \omega}$ and let
	\[ V = \operatorname{stab}_{\Gamma}(v). \]
	Then $V$ is compact and open, and since $\overline{\xi \omega}$ is fixed by $\gamma$,
	\[ \alpha(V) = \operatorname{stab}_{\Gamma}(\gamma \cdot v) = \operatorname{stab}_{\Gamma}(v) = V. \] Hence,  
	\[ s(\alpha) = [ \alpha(V) : V \cap \alpha(V) ] = 1  \]
	in this case too. It follows that $\Gamma$ is uniscalar.
\end{proof}
\begin{theorem}
	\label{thm:uniexception}
	Suppose that $\Gamma$ is an closed subgroup in $\Aff \Tdp$. Then $\Gamma$ is unimodular if and only if it is exceptional.
\end{theorem}
\begin{proof}
	By Proposition \ref{prop:exceptionalimpliesuniscalar}, if $\Gamma$ is exceptional, then it is uniscalar, hence unimodular. The remainder of the argument is the same as the one in Cartwright, Kaimanovich and Woess \cite{cartwright94}. Suppose that $\Gamma$ is non-exceptional. Then $\Gamma \setminus \Hor \Tdp$ is non-empty. It follows that the horocyclic map $\phi$ is non-trivial and that $\phi(\Gamma) = r \I$, where 
	\[ r = \min \left \{ \phi(\gamma_j) > 0 : \{ \gamma \}_{i=1}^k \in \Gamma \right \}. \]
	Let $\gamma$ be an element of $\Gamma$, such that $\phi(\gamma) = r$. It acts by translation on a doubly infinite path $\overline{\xi \omega}$, where $\xi$ is the unique fixed point of $\gamma$ in $\partial^* \Tdp$. Since $\Gamma$ is non-exceptional, $\Gamma$ does not fix $\xi$, hence, there is an element $\alpha$ in $\Gamma$, such that $\alpha \xi \neq \xi$.
	
	Now, $\phi(\alpha) = l r$ for some integer $l$. Let $\beta = \gamma^{-1} \alpha$, and let $\zeta = \beta \xi$. Then, $\phi(\beta) = 0$, and $\zeta$ is in $\partial^* \Tdp \setminus \{ \xi \}$. Let $\xi \wedge \zeta$ be the least common ancestor of $\xi$ and $\zeta$ in $\Td{k}$ . Then, $\gamma \cdot \xi$ and $(\beta \circ \gamma) \cdot x$ are distinct elements, and $\beta$ fixes $x$. Lemma 5 in Woess \cite{woess1991topological} gives that
	\[ \Delta(x) = \frac{|\operatorname{stab}_{\Gamma}(\gamma x)  \cdot  x|}{|\operatorname{stab}_{\Gamma}(x) \cdot (\gamma x) |}. \]
	Therefore,  $\Gamma$ is not unimodular, because $\operatorname{stab}_{\Gamma}(x) \cdot (\gamma x) $ contains $\gamma x$ and $\beta \gamma x$ and $\operatorname{stab}_{\Gamma}(\gamma x)  \cdot  x$ contains only ${x}$, since $\omega_k$ is fixed by $\Gamma$.  	
\end{proof}
\begin{corollary}[Cartwright, Kaimanovich and Woess \cite{cartwright94}]
	\label{cor:nonexceptimpliestransient}
	Suppose that $\mu$ is a probability measure on $\Aff \Tdp$. Suppose that the closed group generated by the support of $\mu$ is non-exceptional. Then $(\Aff \Tdp, \mu)$ is a transient random walk.
\end{corollary}
\begin{proof}
	Since the closed group generated by the support of $\mu$ is non-exceptional it is not unimodular. Non-unimodularity implies transience of the random walk. See Guivarc'h, Keane and Roynette \cite{guivarc77}. 
\end{proof}

\begin{theorem}[Cartwright, Kaimanovich and Woess \cite{cartwright94}, Theorem 4]
	\label{thm:roeafft}
	Suppose that $\mu$ is a probability measure on $\Aff \Tdp$. Let $R_n$ be the right random walk associated with $(\Aff \Tdp, \mu)$.  Suppose $\mu$ has finite first moment with respect to $|\cdot|_{\Tdp}$. Then, 
	\[ \lim_{n \rightarrow \infty} d \left ( R_n o, R_{n+1}o \right ) = 0 \]
	almost surely, and
	\[ \lim_{n \rightarrow \infty} \frac{1}{n} | R_n |_\Tdp = \left | \mm{\pf{\phi}{\mu}} \right |  \]
	almost surely and in $L_1$.
\end{theorem}
\begin{proof}
	The existence of the limit is the same argument as Woess \cite{woess00} Theorem 8.14. If $n$ and $m$ are natural numbers and $n$ is greater than $m$, then 
	\begin{align*}
	|R_n(\omega)|_K = d(e, R_{n+m}(\omega)) &\leq d(e, R_m(\omega)) + d(R_m(\omega) e, R_{n+m}(\omega)) \\
	&\leq d(e, R_m(\omega)) + d( e, X_{m}(\omega) \cdot \ldots \cdot X_{n+m}(\omega)) \\
	&= d(e, R_m(\omega)) + d( e, R_{n}(T^m\omega))
	\end{align*}
	for $\mu^\NN$-almost all $\omega$ in $G^\NN$, where $T$ is the measure preserving left shift map. It follows that, if $\mu$ has finite first moment with respect to $|\cdot|_{\Tdp}$, then the limit, $R_K$, exists and is well defined by applying the Kingman Subadditive Ergodic Theorem in the form given by e.g. Steele \cite{steele1989kingman}. It follows that the limit $\lim_{n \rightarrow \infty} \frac{1}{n} h(R_n o) = \mm{\pf{\phi}{\mu}}$ exists. 
	
	By the Monotone Convergence Theorem,
	\begin{align*}
	0 &= \lim_{n \rightarrow \infty} \frac{1}{n} \int_{\Aff \Tdp} | x | \, d \mu(x) \\
	&=  \lim_{n \rightarrow \infty}  \int_{\Aff \Tdp} \frac{1}{n} | x | \, d \mu(x) \\ 
	&=  \lim_{n \rightarrow \infty}  \int_{\Aff \Tdp} \frac{1}{n} | X_{n+1}(\omega) | \, d \prodmeasure(\omega) \\ 
	&=  \lim_{n \rightarrow \infty}  \int_{\Aff \Tdp} \frac{1}{n} d \left ( R_n(\omega) o, R_{n+1}(\omega) o \right ) \, d \prodmeasure(\omega). 
	\end{align*}
	Since $d$ is always non-negative, $\lim_{n \rightarrow \infty} d \left ( R_n(\omega) o, R_{n+1}(\omega)o \right )$ is zero almost surely. Therefore, $R_n o$ is regular by Proposition \ref{prop:regularity} and $|R_n |$ converges to $|\mm{\pf{\phi}{\mu}}|$ by Lemma \ref{lem:regularity}.
\end{proof}

\begin{theorem}[Cartwright, Kaimanovich and Woess \cite{cartwright94}, Theorem 2]
	\label{theo:conv}
	Let $\mu$ be a probability measure on $\Aff \Tdp$. Suppose that the closed group generated by the support of $\mu$ is non-exceptional. Let $\phi$ be the horocyclic map from $\Aff \Tdp$ to the integers. Let $\{ R_i \}_{i \in \NN}$ be the right random walk associated with $(\Aff \Tdp, \mu)$ Let $o$ be a distinguished vertex of $\Tdp$, and let $\omega$ be the fixed end. 
	\begin{enumerate}[(i)]
		\item If $m_1 (\pf{\phi}{\mu})$ is finite and the mean of $\pf{\phi}{\mu}$ is negative, then $R_n o$ converges almost surely to $\omega$.
		\item If $m_1 (\mu)$, with respect to $| \cdot |_{\Tdp}$, is finite and the mean of $\pf{\phi}{\mu}$ is positive, then $R_n o$ converges to a random end in $\partial^* \Tdp$ almost surely.
		\item If $m_1 (\mu)$, with respect to $| \cdot |_{\Tdp}$, is finite and the mean of $\pf{\phi}{\mu}$ is zero, and
		\[ \EV \left ( \left | o \wedge R_1^{-1} o \right | q^{|o \wedge R_1 o|} \right ) \]
		is finite then $R_n o$ converges to $\omega$ almost surely.
	\end{enumerate}
\end{theorem}

\section{Products of affine automorphism groups of homogeneous trees}

Let $P = \prod_{i=1}^k \Aff \Td{i}$ be a finite direct product of affine automorphism groups of trees with the product topology. For each affine group $\Aff \Td{j}$, let $\omega_j$ be the fixed end, let $o_j$ be the distinguished vertex and let $d_j$ be the vertex degree. Let $\phi_j$ be the horocyclic map from $\Aff \Td{j}$ to $\I$. Let $\omega = \{ \omega_i \}_{i=1}^k$, $o = \{ o_i \}_{i=1}^k$ and $\phi = \prod_{i=1}^k \phi_i $. Let $\pi_j$ be the projection map from $P$ to $ \Aff \Td{j}$. Given an element $\alpha$ in $P$, write $\alpha_j$ to mean $\pi_j(\alpha)$. Similarly, if $v$ is in $\prod_{i=1}^k \Td{i}$, we write $v_j$ to mean the projection of $v$ to $\Td{j}$.

Notice that $P$ is $\sigma$-compact because it is a product of $\sigma$-compact spaces, and that the topology is second countable, locally compact and totally disconnected.  The action of each factor $\Aff \Td{j}$ on $\partial \Td{j} \cup \Td{j} $ extends to an action of $P$ on the Cartesian product $\prod_{i=1}^k \left (\partial \Td{i} \cup \Td{i} \right )$.
Denote by $\lambda_P$ the right Haar measure on $P$ normalised so that compact open subgroups of the form
\[ \operatorname{stab}_{P}(v) = \{ \beta \in P, \ \beta_i{v_i} = v_i \ \forall 0 < i \leq k \}, \]
have measure $1$, where $v = \{ v_i \}_{i=1}^k$, and $v_j$ is in $\Td{j}$ for each natural number $j$. Since $P$ is not compact, $\lambda_{P}$ is an infinite measure.

\subsection{Partially exceptional and fully exceptional subgroups}
Suppose that $\Gamma$ is a closed subgroup in $P$.  If $\pi_j(\Gamma)$ is a exceptional subgroup of $\pi_j(P)$ for every natural number $j$ less than or equal to $k$, then we say that  $\Gamma$ is \emph{fully exceptional}. If there is at least one natural number $j$, such that $\pi_j(\Gamma)$ is a exceptional subgroup of $\pi_j(P)$, then we say that $\Gamma$ is \emph{partially exceptional}. Obviously, if $\Gamma$ is fully exceptional then it is partially exceptional. A partially exceptional subgroup need not be uniscalar nor unimodular.

 \begin{proposition}
	Let $\Gamma$ be an closed subgroup in finite product of affine automorphism groups of trees, $P = \prod_{i=1}^k \Aff \Td{i}$. Then, $\Gamma$ is fully exceptional if and only if it is uniscalar.
\end{proposition}
\begin{proof}
	Suppose that $\Gamma$ is fully exceptional. Let $\gamma$ be an automorphism of $\Gamma$. Let $H$ be the set of natural numbers $i$, such that $\pi_i(\Gamma)$ is contained in $\Hor \Td{i}$. Let $H'$ be the set of natural numbers $i$ not in $H$, which are less than or equal to $k$. Then, $\gamma_i $ fixes the element $v_i := o_i \wedge \gamma_i o_i$ in $\Td{i}$ for all $i$ in $H$.
	For each $i$ in $H$, let $\nu_i$ in $\partial^* \Td{i}$ be the element fixed by $\pi_i(\Gamma)$ and let $v_i$ be a vertex in $\overline{\nu_i \omega_i}$. Then, $v_i$ is translated by the action of $\gamma_i$ for each $i$ in $H'$. Let 
	\[ V = \operatorname{stab}_{\Gamma}\left ( v \right ), \]
	where $v :=  \{v_i\}_{i=1}^k $ is in $\prod_{i=1}^k \Td{i}$.
	
	Obviously, $V$ is compact and open. For each $i$ in $H'$, let $g$ be the element in $P$, such that $\gamma_i (x) = g_ixg_i^{-1}$ for every $x$ in $\Aff \Td{i}$. For each $i$ in $H$, let $g_i$ be the identity element in $\Aff \Td{i}$. Then, since $\gamma_i$ fixes $\overline{\nu_i \omega_i}$ for each $i$ in $H'$, 
	\[ \gamma(V) = \operatorname{stab}_\Gamma(g \cdot v) = \operatorname{stab}_\Gamma(v) = V, \]
	where $g$ acts pointwise on $v$. It follows that
	\[ s(\gamma) = [ \alpha(V) : V \cap \alpha(V) ] = 1.  \]
	Since $\gamma$ was an arbitrary automorphism, $\Gamma$ is uniscalar.
	
	Suppose instead that $\Gamma$ is not fully exceptional. We will construct a compact open subgroup $U$ which is tidy for an automorphism $\alpha$ of $P$ with scale greater than one. Let $E$ be the set of natural numbers $i$ less than or equal to $k$ for which $\pi_i(\Gamma)$ is non-exceptional. Then $E$ is non-empty. Let $E'$ be the set of natural numbers $i$ less than or equal to $k$ not in $E$, so that $\pi_i(\Gamma)$ is exceptional. As before, $H$ be the set of natural numbers $i$, such that $\pi_i(\Gamma)$ is contained in $\Hor \Td{i}$ and let $H'$ be the set of natural numbers $i$ in $E'$ and not in $H$.
	
	Let $j$ be a particular element of $E$. Then, $\pi_j(\Gamma)$ is non-exceptional, $\phi_j (\Gamma)$ is a non-trivial homomorphism, and $\phi_j (\Gamma) = r_j \I$, where
	\[ r_j = \min \left \{ \phi_j(\gamma_j) > 0 : \{ \gamma_i \}_{i=1}^k \in \Gamma \right \}. \]
	Choose $\gamma $ in $\Gamma$, such that $\phi_j (\gamma) = -r_j$. Then $\gamma$ acts by translation on $\Td{i}$. Let $v_j$ be a vertex in $\Td{j}$ on the axis of translation of $\gamma_j$. For the natural numbers $i$ in $E \setminus \{j\}$, $\gamma_i$, choose $v_i$ to be the vertex in $\Td{i}$ if $\gamma_i$ is elliptic, and choose $v_i$ to be a vertex in $\Td{i}$ on the axis of translation if $\gamma_i$ is hyperbolic. For each $i$ in $E'$, let $v_i := o_i \wedge \gamma_i o_i$ in $\Td{i}$, which is fixed by the action of $\gamma$. If $i$ is in $E$. For each $i$ in $E$, let $v_i$ be a vertex on the axis fixed by $\phi_i(\Gamma)$. Let $\alpha$ be conjugation by $\gamma$, and let $U$ be the compact open subgroup given by
	\[ U =  \operatorname{stab}_{\Gamma}\left ( v \right ), \]
	where $v :=  \{v_i\}_{i=1}^k $.	Then,
	%
	%
	\[ U \cap \alpha(U) = \operatorname{stab}_P \{ x_i \}_{i=1}^k,  \]
	where
	\[ x_r = \begin{cases} v_r &\textrm{if} \ r \in E, \\ v_r &\textrm{if} \ r \in E' \ \textrm{and} \ \phi_r(\gamma) \geq 0, \\
	\gamma_r v_r &\textrm{otherwise}   \end{cases} \]
	for each natural number $r$ less than or equal to $k$. Hence, 
	\[ [\alpha(U) : U \cap \alpha(U)] \geq d_j \]
	by choice of $\gamma$. Since $\omega_i$ is fixed by $\phi_i(\Gamma)$ for each natural number $i$ less than or equal to $k$, $U = U_+ U_-$, and $U$ is tidy above. It is tidy below because $U_{++}$ is a union of open sets. Hence, $\Gamma$ is not uniscalar. 	
\end{proof}
\begin{proposition}
	\label{prop:uniscalartransient}
	Suppose that $\mu$ is a probability measure on the finite product of affine automorphism groups of trees, $P = \prod_{i=1}^k \Aff \Td{i}$. Suppose that the closed group generated by the support of $\mu$, $\gr \mu$, is not fully exceptional. Then, $(P, \mu)$ is transient, i.e. the right random walk associated with $(P, \mu)$ eventually leaves every compact set $E$ almost surely.
\end{proposition}
\begin{proof}
	Since $\gr \mu$ is not fully exceptional, there is a natural number, $j$, such that $\pi_j(\gr \mu)$ is a non-exceptional subgroup of $\Aff \Td{j}$. 	Let $E$ be a compact set in $\gr \mu$. Then $\pi_j(E)$ is compact because $\pi_j$ is continuous and $\Gamma$ is Hausdorff. Corollary \ref{cor:nonexceptimpliestransient} implies that the projected random walk $(\Aff \Td{j}, \pf{(\pi_j)}{\mu})$ is transient., i.e. the right random walk corresponding to $(\Aff \Td{j}, \pf{\pi_j}{\mu})$ leaves $\pi_j(E)$ almost surely after finitely many steps.  It follows that the right random walk associated with $(P, \mu)$ eventually leaves $E$ almost surely.
\end{proof}
\begin{lemma}
	\label{lem:unimoduniscale}
	A closed subgroup $\Gamma$ in $P$ can be unimodular but not uniscalar. 
\end{lemma}
\begin{proof}
	We demonstrate how to construct an example for the case that $k=2$ and both trees have the same degree $d$. The same construction type of construction can be used if $k$ is larger and/or the trees don't have the same vertex degree.
	
	Let $\nu_1$ be an element of $\partial^* \Td{2}$ and let $\nu_2$ be an element of $\partial^* \Td{2}$. Let $\sigma_1$ be the element in $\Aff \Td{1}$ which acts as a translation along $\overline{\nu_1 \omega_1}$ away from $\omega_1$ with $\phi_1 (\sigma_1) = 1$. Similarly,  $\sigma_2$ be the element in $\Aff \Td{2}$ which acts as a translation along $\overline{\nu_2 \omega_2}$ with $\phi_2 (\sigma_2) = -1$ towards $\omega_1$. Let
	\[ \Gamma = \left \langle (\sigma_1, \sigma_2^{-1}) \right \rangle \]
	Lemma 5 in Woess \cite{woess1991topological} gives that
	\[ \Delta_P(\sigma_1, \sigma_2) = \frac{|\operatorname{stab}_G (o_1,  o_2)  \cdot  (\sigma_1 o_1, \sigma_2 o_2)|}{|\operatorname{stab}_G (\sigma_1 o_1, \sigma_2 o_2)  \cdot  (o_1, o_2)|}  = 1. \]
	As $\Delta_P$ is a homomorphism, it follows that $\Gamma$ is a unimodular subgroup. To see that $\Gamma$ is not uniscalar, let $\alpha = (\alpha_1, \alpha_2)$ be conjugation by $(\sigma_1, \sigma_2^{-1})$. Let $O$ be the stabilizer in $P$ of a pair of vertices $(o_1, o_2)$. Then $\alpha(O) = \operatorname{stab}_P (\sigma_1 o_1, \sigma_2^{-1} o_2)$. Since $\omega_1$ and $\omega_2$ are both fixed,
	%
	\begin{align*}
	O \cap \alpha(O) 
	&= \operatorname{stab}_P (\sigma_1  o_1, o_2) 	    
	\end{align*}
	Hence,
	\begin{align*}
	[\alpha(O) : O \cap \alpha(O)] = d, 
	\end{align*}
	which by assumption on $d$ is greater than one. We need to check that $O$ is tidy above and tidy below. Since $(\omega_1, \omega_2)$ is fixed, $O = O_+ O_-$ and $O$ is tidy above. It is tidy below because the subgroup
	\[ O_{++} = \bigcup_{k=0}^\infty \operatorname{stab}_P (\sigma_1^k o_1, \sigma_1^k o_2) \]
	is a union of open sets, which is therefore closed. It follows that $\alpha$ has scale $d$, and $\Gamma$ is not uniscalar. 
\end{proof}

\subsection{Random walks and gauges}
Suppose that $\mu$ is an aperiodic, spread-out probability measure on $P$ that has finite first moment with respect to $| \cdot |_P $. Suppose that the closed semigroup generated by the support of $\mu$ is a closed subgroup $\Gamma$ of $P$ which is not fully exceptional.  Let $\{ R_i \}_{i=1}^\infty$ be the right random walk associated with $(P, \mu)$.   For each natural number $j$ less than or equal to $k$, let $\imesi{j}$ be the pushforward measure 
\[ \imesi{j} (E) = \pf{\phi_j }{\mu}. \]
For each element $\varphi$ in $P$, let
\[ | \varphi |_P = \sum_{i=1}^k d(o_i, \varphi_i o_i) = \sum_{i=1}^k | \varphi_i |_{\Tdp{i}}. \]
The map $| \cdot |_{P}$ is a subadditive gauge function, where the sequence
\[ \mathcal{A}^P = \left \{ \left \{ \gamma \in \Gamma : \sum_{i=1}^k d(o_i, \gamma_i o_i) \leq j \right \} \right \}_{j=1}^\infty. \]
is the corresponding subadditive gauge.
\begin{lemma}
	\label{lem:ffmrwp}
	Suppose that $\mu$ a probability measure on $P$. Let $j$ be any natural number less than or equal to $k$.  If $\mu$ has finite first moment with respect to $|\cdot|_P$, then $\pf{\pi_j}{\mu}$ has finite first moment with respect to $| \cdot |_{\Td{j}}$ and $\mu_j$ has finite first moment with respect to the ordinary absolute value on the integers.
\end{lemma}
\begin{proof}
	If $m_1 (\mu)$ is finite, then $m_1 (\pf{\pi_j}{\mu})$ is finite, since
	\begin{align*}
	m_1 (\mu) &= \int_\Gamma |x|_P \, d \mu(x) \\
	&= \int_\Gamma \sum_{i=1}^k | x_i |_{\Td{i}} \, d \mu(x)  \\
	&= \sum_{i=1}^k \int_\Gamma  | x |_{\Td{i}} \, d \pf{\pi_i}{\mu}(x) \\
	&= \sum_{i=1}^k m_1 (\pf{\pi_i}{\mu}) \\
	&\geq m_1 (\pf{\pi_j}{\mu}).
	\end{align*}
	The argument $m_1 (\mu_j)$ is finite whenever $m_1 (\mu)$ is finite is trivial.
\end{proof}
\begin{corollary}
	\label{cor:productlimit}
	Suppose that $\mu$ is a probability measure on $P$. Let $\{ R_i \}_{i \in \NN}$ be the right random walk associated with $(P, \mu)$. Suppose that $\mu$ has finite first moment with respect to $|\cdot|_P$, and that $j$ is a natural number less than or equal to $n$. 
	\begin{enumerate}[(i)]
		\item If $m_1(\mu_j)$ is negative, then $\{ \pi_j(R_i) o_j \}_{i=1}^\infty$ converges almost surely to $\omega$.
		\item If $m_1(\mu_j)$ is positive, then $\{ \pi_j(R_i) o_j\}_{i=1}^\infty$ converges to a random end in $\partial^* \Td{j}$.
	\end{enumerate}
\end{corollary}
\begin{proof}
	This is a special case of Theorem \ref{theo:conv}.
\end{proof}

Let $B = \prod_{i=1}^k \partial \Td{i}$ with the product topology, where each factor $\partial \Td{j}$ has the subspace topology from the ultrametric on $\hat{\Td{j}}$. Then, $B$ is compact, because each factor $\partial \Td{j}$ is closed in the compact set $\hat{\Td{j}}$. Let $P$ act on $B$ with the continuous action satisfying
\[ \pi_j ( g \cdot b ) = \pi_j(g) \cdot \pi_j(b) \]
for each $g$ in $P$, $b$ in $B$ and natural number $j$ less than or equal to $k$, where $\pi_j(g)$ in $\Aff \Td{j}$ acts on $\pi_j(b)$ in $\partial \Td{j}$ by taking the restriction to $\partial \Td{j}$ of the extension of $\pi_j(g)$ to a the compatible homeomorphism of $\hat{\Td{j}} = \Td{j} \cup \partial \Td{j}$.  Let $R_\infty$ be a map from the path space $\Gamma^\NN$ to $B$ satisfying
\[ \pi_j (R_\infty(\gamma)) = \begin{cases} \omega_j & \textrm{if} \ m_1(\mu_j) \geq 0, \\
\lim_{n \rightarrow \infty} \pi_j(R_n(\gamma)) o_j & \textrm{if} \ m_1(\mu_j) < 0. 
\end{cases} \]
for almost all paths $\gamma$. The limit exists almost everywhere by Corollary \ref{cor:productlimit}. Let $\nu = \pf{(R_\infty)}{\pathmeasure}$ be the hitting measure on $B$. Then $\nu$ is  $\mu$-stationary, because
\begin{align*}
\int_P \int_B f(xb) \, d \nu(b) \, d \mu(x) &= \int_P \int_{P^\NN} f(x R_\infty(\gamma) ) \, d \pathmeasure(\gamma) \, d \mu(x) \\
&= \int_{P^\NN} f( R_\infty(T \gamma) ) \, d \pathmeasure( \gamma)  \\
&= \int_{P^\NN} f( R_\infty(\gamma) ) \, d \pathmeasure( \gamma)  \\
&= \int_B f(b) \, d \nu(b),
\end{align*}
where $T$ is the left shift map and $f$ is any function in $C(B)$. Notice that $R_\infty$ is shift invariant in the sense used by Kaimanovich in \cite{kaimanovich00} or  Brofferio and Schapira \cite{brofferio2011poisson} Proposition 2.1. It follows that $(B, \nu)$ is a $\mu$--boundary of $(P,\mu)$. 
\begin{proposition}
	\label{prop:sequniformlytemperategauges}
	Suppose that $\mu$ is an aperiodic, spread-out probability measure on $P$ that has finite first moment with respect to $| \cdot |_P $. Suppose that the closed semigroup generated by the support of $\mu$ is a closed subgroup $\Gamma$ of $P$ which is not fully exceptional. Let $u$ be an element of $B$ and let $n$ be a natural number. Let
	\[ \mathcal{A}^{(n)}  (u) = \left \{ \left \{ \gamma \in \Gamma : \sum_{i=1}^k d \left ( \left (  \overline{o_i u_i} \right )_{\flr{n m_1(\mu_i)}}, \gamma_i o_i \right ) \leq j \right \} \right \}_{j=1}^\infty. \]
	Then, $\{ \mathcal{A}^{(r)} (u) \}_{r=1}^\infty$ is a uniformly temperate sequence of gauges on $\Gamma$, and
	\[ \left | \varphi \right |_{\mathcal{A}^{(n)} (u)} =  \sum_{i=1}^k d \left ( \left (  \overline{o_i u_i} \right )_{\flr{n m_1(\mu_i)}}, \gamma_i o_i \right ) \]
	is the corresponding gauge function on P. 
\end{proposition}
\begin{proof}
	Each gauge is well defined because the finite first moment condition and Lemma \ref{lem:ffmrwp} ensures that $m_1(\mu_i)$ is always finite.	If $\gamma$ is an element of $\Gamma$, then the distance from $\gamma_i o_i$ to any vertex is finite. It follows that
	\[ \Gamma = \bigcup_{j=1}^\infty  \left \{ \gamma \in \Gamma : \sum_{i=1}^k d \left ( \left (  \overline{o_i u_i} \right )_{\flr{n m_1(\mu_i)}}, \gamma_i o_i \right ) \leq j \right \}. \]
	Let $\gamma$ and $\phi$ be elements of $\Gamma$. The gauge map corresponding to  $\mathcal{A}^{(n)}  (u)$ is
	%
	\begin{align*}
	| \gamma |_{\mathcal{A}^{(n)}(u)} &= \min \left \{j \in \NN: \gamma \in \mathcal{A}_j \right \} \\
	&= \sum_{i=1}^k d \left ( \left (  \overline{o_i u_i} \right )_{\flr{n m_1(\mu_i)}}, \gamma_i o_i \right ).
	\end{align*} 
	Let $K$ be the non-negative constant given by
	\[ K = | e |_{\mathcal{A}^{(n)}} = \sum_{i=1}^k d \left ( \left (  \overline{o_i u_i} \right )_{\flr{n m_1(\mu_i)}}, o_i \right ) \]
	where $e$ is the identity element in $\Gamma$. By repeated use of the triangle inequality, 
	\begin{align*}
	| \gamma \phi |_{\mathcal{A}^{(n)}} &= \sum_{i=1}^k d \left ( \left (  \overline{o_i u_i} \right )_{\flr{n m_1(\mu_i)}}, \gamma_i \phi_i o_i \right ) \\
	&\leq \sum_{i=1}^k d \left ( o_i, \gamma_i \phi_i o_i \right ) + K \\
	&= \sum_{i=1}^k d \left ( \gamma_i^{-1} o_i, \phi_i o_i \right ) + K \\
	&\leq \sum_{i=1}^k d \left ( \gamma_i^{-1} o_i, o_i \right ) + \sum_{i=1}^k d \left ( o_i, \phi_i o_i \right ) + K \\
	&= \sum_{i=1}^k d \left ( o_i, \gamma_i o_i \right ) + \sum_{i=1}^k d \left ( o_i, \phi_i o_i \right ) + K \\
	&= | \gamma |_{\mathcal{A}^{(n)}} + | \phi |_{\mathcal{A}^{(n)}} + K.
	\end{align*}
	So $| \cdot |_{\mathcal{A}^{(n)}}$ is a gauge function. For each natural number $j$, let
	\[ M_j = \lambda_\Gamma \left ( \left \{ \gamma \in \Gamma : \sum_{i=1}^k d \left ( \left (  \overline{o_i u_i} \right )_{\flr{n m_1(\mu_i)}}, \gamma_i o_i \right ) \leq j \right \} \right ).  \]
	where $\lambda_\Gamma$ is a right Haar measure on $\Gamma$, normalised so the compact open stabilizer of an element in $\prod_{i=1}^k \Td{i}$ has measure $1$. From the definition, there is a real constant $C$, such that 
	\[ M_j \leq   C \left ( \max_{i=1}^k \{ d_i \} \right )^j \]
	where $d_i$ is the degree of $\Td{i}$. Hence, the sequence$\{ \mathcal{A}^{(r)} (u) \}_{r=1}^\infty$ is uniformly temperate.
\end{proof}

\begin{proposition}
	\label{prop:uncountablesupportpositivedrift}
	Suppose that $\mu$ is an aperiodic, spread-out probability measure on $P$ that has finite first moment with respect to $| \cdot |_P $. Suppose that  $\Gamma = \sgr \mu$ is a closed subgroup of $P$ which is not fully exceptional. Let $(B,\nu)$ be the \PF boundary. If there is a natural number $j$ less than or equal to $k$, such that
	\[ m_1(\mu_j) > 0, \]
	then $\nu$ is a continuous measure supported on an uncountable set.
\end{proposition}
\begin{proof}
	We adapt the arguments of Proposition 2 and Theorem 3 in Cartwright, Kaimanovich and Woess \cite{cartwright94} to the product of trees case.
	
	Let $\partial \Gamma$ be the set of accumulation points of an orbit $\Gamma v$ in $B$ for some $v$ in $\prod_{i=1}^k \Td{i}$. Since $\Gamma$ is a subgroup, this orbit is not dependent on the particular choice of $v$. Choose any natural number $j$, such that $ m_1(\mu_j) > 0$. Let $\partial_j \Gamma$ be the set of accumulation points of an orbit $\pi_j (\Gamma) v_j$ in $\pi_j(B)$. Since $\pi_j (\Gamma)$ is non-exceptional, there exists an element $\alpha$ in $\Gamma$, such that 
	\[ \phi_j(\alpha) = \min \{ \phi_j (\gamma) \geq 0 : \gamma \in \Gamma \}. \]
	Then $\alpha_j$ is hyperbolic, and so $\alpha$ acts on $\Td{j}$ by translation along a doubly infinite sequence $\overline{u_j \omega_j}$, where $u$ is an element of $B$ distinct from $\omega$. Since $\pi_j(\Gamma)$ is non-exceptional, there is a $\beta$ in $\Gamma$, such that $\beta_j u_j$ is an end in $\partial \Td{j}$ distinct from $u_j$. Let $\xi = \beta \alpha \beta^{-1}$. Choose any sequence of natural numbers $k = \{ k_i \}_{i=1}^\infty$. For each natural number $j$, let
	\[ x_j^{(k)} = \prod_{i=1}^j \alpha^{k_i} \xi^{k_i}. \]
	Choose $v$ to be the common ancestor of $u_j$ and $\beta_j u_j$ in $\Td{j}$. The sequence $x_i^{(k)} o_j$ is convergent to an end for each choice of $k$, because it is Cauchy with respect to the ultrametric $\theta$, and $(x_i^{(k)} \curlywedge x_{i+1}^{(k)}) v$ tends to $\infty$ as $i$ does.	Furthermore, each distinct choice of $k$ results in convergence to a distinct end in $\partial \Td{j}$,  because the sequences are eventually contained in disjoint open balls. We conclude $\partial_j \Gamma$ is uncountable.	Since $B$ is a closed subset of a product of compact second countable spaces, there is a subsequence $\{ x_{n_i}^{(k)} \}_{i=1}^\infty$, such that $\{ x_{n_i}^{(k)} o_r \}_{i=1}^\infty$ is convergent for all natural numbers $r$ less than or equal to $k$. Hence, $\partial \Gamma$ is uncountable too.
	
	Suppose that $\alpha = \{\alpha^{(i)} \}_{i=1}^\infty$ is a sequence in $\pi_j(\Gamma)$, such that $\alpha^{(i)} o_j$ converges to an element $v$ in $\partial \Td{j} \setminus \{ \omega_j \}$. Let 
\[ \beta^{(i)} = \left ( \alpha^{(i)} \right )^{-1}. \]
	Then $\beta^{(i)} o_j$ converges to $\omega_j$ and $\omega_j$ is in $\partial_j \Gamma$. 	Taking $\varphi_r = \beta^{(r)}$, $\xi = \omega_j$ and $\eta = v$, Lemma 2.2 in Cartwright and Soardi \cite{cartwright89} gives that $\beta^{(i)}_j u$ converges to $\omega_j$ for every $u$ in $\partial \Td{j} \setminus \{ \omega, v \}$. Hence, $\omega_j$ is an accumulation point of $\pi_j ( \Gamma ) u$ for every $u$ in $\partial \Td{j} \setminus \{ \omega_j \}$. Using Lemma 2.2 in Cartwright and Soardi again with  $\varphi_r = \alpha^{(r)}$, $\xi = v$ and $\eta = \omega_j$ gives that $\beta^{(i)}_j u$ converges to $v$ for every $u$ in $\partial \Td{j} \setminus \{ \omega_j \}$. Given some $u$ in $\partial \Td{j} \setminus \{ \omega_j \}$, choose $\eta$ in $\Gamma$, such that $\eta_j u$ and $u$ are distinct ends. Either $\{ \alpha^{(i)} u \}_{i=1}^\infty$ or $\{ \alpha^{(i)} \eta u \}_{i=1}^\infty$ has a subsequence of elements distinct from $v$. Hence, $v$ is a limit point of $\Gamma u \setminus \{ v \}$. Thus, for each $u$ in $B$, such that $u_j$ is not equal to $ \omega_j $ the orbit $\pi_j (\Gamma) u_j$ is dense in $\partial_j \Gamma$.
	
	Since $ m_1(\mu_j) > 0$, $\nu(\pi^{-1}(\{\omega_j\}))$ is zero, because $\pi_j ( R_\infty )$ is in $\partial^* \Td{j} \setminus \{ \omega \}$ almost surely by Theorem	\ref{theo:conv}. Let $M = \max \left \{  \nu(\{ u \}) : u \in B \right \}.$ Suppose that $M$ is positive. Then the set
	\[ S = \{ u \in B : \{  \nu(\{ u \})  = M \} \]
	is non-empty and finite, because $\nu$ is a probability measure. Since $\nu$ is $\mu$-stationary, $\gamma S = S$ for all $\gamma$ in $\supp \mu$. But then $\Gamma S = S$. So $\pi_j(\Gamma) \pi_j(S) = \pi_j(S)$, which is a contradiction, because $\pi_j(\Gamma) \pi_j(S)$ is dense in $\partial_j \Gamma$.
	
\end{proof}

\begin{theorem}
	\label{thm:poisonnboundaryofproduct}
Suppose that $\mu$ is an aperiodic, spread-out probability measure on $P$ that has finite first moment with respect to $| \cdot |_P $. Suppose that  $\Gamma = \sgr \mu$ is a closed subgroup of $P$ which is not fully exceptional. Then, $(B, \nu)$ is the \PF boundary of $(P, \mu)$. 
\end{theorem}
\begin{proof}
	Suppose that $j$ is a natural number less than or equal to $k$. Let $u$ be in $B$, let $n$ be a natural number and let $\mathcal{A}^{(n)} (u)$ be the gauge from Proposition \ref{prop:sequniformlytemperategauges}, 
	\[ \mathcal{A}^{(n)} (u) = \left \{ \left \{ \gamma \in \Gamma : \sum_{i=1}^k d \left ( \left (  \overline{o_i u_i} \right )_{\flr{n m_1(\mu_i)}}, \gamma_i o_i \right ) \leq r \right \} \right \}_{r=1}^\infty. \]
	Theorem \ref{thm:roeafft} gives that 
	\[ \lim_{n \rightarrow \infty} d \left ( \pi_j(R_n) o_j, \pi_j(R_{n+1})o_j \right ) = 0 \]
	almost surely, and
	\[ \lim_{n \rightarrow \infty} \frac{1}{n} | \pi_j(R_n) | = \left | \mm{\pf{\phi}{\mu}} \right |  \]
	almost surely and in $L_1$. In particular, $\pi_j(R_n) o_j$ is almost surely a regular sequence of vertices in $\Td{j}$, with rate of escape $|\mm{\imesi{j}}|$. Thus, for almost every path $\psi$ in the random walk, there is an end $u_j (\psi)$ in $\partial \Td{j}$, such that
	\[ \lim_{n \rightarrow \infty}  d \left ( \pi_j(R_n (\psi)) o_j, (\overline{o_j u_j (\psi)})_{\flr{m_1 (\mu_j) n}} \right ) = |\mm{\imesi{j}}|.  \]
	If  $|\mm{\imesi{j}}| \neq 0$, then $\pi_j(R_n) o_j \rightarrow u_j (\psi)$, so $u_j (\psi)$ is unique. By Theorem \ref{theo:conv}, $u_j (\psi) = \omega_j$ if $\mm{\imesi{j}} < 0$ and $u_j (\psi)$ in $\partial^* \Td{p}$ if $\mm{\imesi{j}} > 0$. If $|\mm{\imesi{j}}| = 0$, then $u_j (\psi)$ is arbitrary and the statement is just that
	\[ \lim_{n \rightarrow \infty}  d \left ( \pi_j(R_n (\psi)) o_j, o_j \right ) = 0.  \]
	Hence, for almost all paths $\gamma$, 
	\begin{align*}  
	\lim_{n \rightarrow \infty} \frac{1}{n} \left | R_n(\gamma)  \right |_{\mathcal{A}^{(n)} (R_\infty (\gamma))}	
	&= 
	\lim_{n \rightarrow \infty} \frac{1}{n} \sum_{i=1}^k d \left ( \pi_i(R_n(\gamma)) o_i, \left ( \left (\overline{o_i  \left (R_\infty(\gamma) \right )_i} \right )_{\flr{m_1 (\mu_i) n}} \right ) \right ) \\ &
	= \lim_{n \rightarrow \infty} \frac{1}{n} \sum_{i=1}^k |\mm{\imesi{i}}| \\
	&= 0.
	\end{align*}
	By assumption, $\mu$ has finite first moment with respect to the subadditive gauge $\mathcal{A}^P$. Hence, by Kaimanovich's ray criterion for topological groups, Theorem \ref{thm:kaimanovichapproxthmtop}, $(B, \nu)$ is the \PF boundary.
\end{proof}
\begin{corollary}
Suppose that $\mu$ is an aperiodic, spread-out probability measure on $P$ that has finite first moment with respect to $| \cdot |_P $. Suppose that  $\Gamma = \sgr \mu$ is a closed subgroup of $P$ which is not fully exceptional. The \PF boundary of $(P, \mu)$ is trivial if and only if 
	\[ m_1(\mu_j) \leq 0 \]
	for all natural numbers $j$ less than or equal to $k$. 
\end{corollary}
\begin{proof}
	If $ m_1(\mu_j) \leq 0 $ for all natural numbers $j$ less than or equal to $k$, then it follows from the definition that the boundary is a single point.	If there is a natural number $j$, such that $ m_1(\mu_j) > 0 $, then Proposition \ref{prop:uncountablesupportpositivedrift} implies that $\nu$ is not a point measure. 
\end{proof}

\begin{proposition}
	\label{prop:transcompactgen}
	Let $\Gamma$ be an closed subgroup in finite direct product of affine automorphism groups of trees, $P = \prod_{i=1}^k \Aff \Td{i}$. Suppose that the action of $\Gamma$ on the product of trees $\prod_{i=1}^k \Td{i}$ is transitive. Then, $\Gamma$ is generated by the compact neighbourhood of the identity
	\[ J = \left \{ \alpha \in \Gamma : |\alpha|_P \leq 1 \right \}. \]
\end{proposition}
\begin{proof}
	By definition, $J$ is contained in $\Gamma$. Since $\Gamma$ is transitive, there is an element $\sigma^{(i)}$ in $\Gamma$, for every natural number $i$ less than or equal to $k$, such that
	\[ \sigma^{(i)}_j o_j \begin{cases} \left ( \overline{o_j \omega_j} \right )_1 &\textrm{if} \ i = j,\\
	o_i & \textrm{otherwise}.   \end{cases} \]
	Furthermore, each $\sigma^{(i)}$ is in $J$, because
	\[ | \sigma^{(i)} |_P = \sum_{i=1}^k | \sigma^{(i)}_j  |_{\Td{i}} = 1. \] 
	Let $\alpha$ be an element of $\Gamma$. We will show that $\alpha$ is a product of elements in $J$. Let 
	\[ \beta = \left (  \prod_{i=1}^k \left ( \sigma^{(i)} \right )^{-\phi_i (\alpha)} \right ) \alpha.   \]
	For brevity, let \[ \Omega = \left ( \prod_{i=1}^k \left ( \sigma^{(i)} \right )^{h(o_i \wedge \beta_i o_i)} \right ),\] which is a project of elements in $J$. Let $j$ be a natural number less than or equal to $k$. Then, $\beta_j$ is in $\Hor \Td{j}$, so it fixes $o_j \wedge \beta_j o_j$. Therefore, 
	\[ \pi_j \left ( \Omega  \beta \Omega^{-1} \right ) \]
	is in $\Stab{\pi_j(\Gamma)}{o_j}$. Notice that the set 
	\[ \left \{ \alpha \in \Gamma : \alpha_i \in \Stab{\pi_i(\Gamma)}{o_i} \ \textrm{for all natural numbers $i$ less than or equal to $k$} \right \} \]
	is contained in $J$.  It follows that  
	%
	\[ \beta =  \Omega^{-1} \Omega \beta \Omega^{-1} \Omega  \] 
	%
	is a product of elements in $J$. Since  
	\[ \alpha = \left (  \prod_{i=1}^k \left ( \sigma^{(i)} \right )^{-\phi_i (\alpha)} \right )^{-1} \beta,   \]
	$\alpha$ is also a product of elements in $J$. 
\end{proof}

\begin{remark}
	Suppose $\Gamma$ and $\mu$ are as in Theorem \ref{thm:poisonnboundaryofproduct}, and that $\Gamma$ acts transitively on  $  \prod_{i=1}^k \Td{i}$. Choose an element $\gamma$ in $\Gamma$ for every  $x = \{ x_i \}_{i=1}^k$ in $  \prod_{i=1}^k \Td{i}$, such that
	\[ \gamma(x_i)o_i = x_i. \]
	For each element $u$ in $B $, let
	\[ \Pi_n (u) = \gamma_n \left ( x^{(n)} \right ) \]
	where $x^{(n)} = \left \{ x_i^{(n)} \right \}_{i=1}^k$ and 
	\[ x_p^{(n)} = \left ( \overline{o u_p} \right )_{\flr{n m_1(\mu_p)}}. \]
	Let $p$ be a natural number less than or equal to $k$. Notice that if $m_1(\mu_p) = 0$, then $ x_p^{(n)} = o$ for all natural numbers $n$. Hence, for almost all paths $\gamma$, 
	\begin{align*}  
	\lim_{n \rightarrow \infty} \frac{1}{n} \left | R_n^{-1}(\gamma) \Pi_n (R_\infty (\gamma)) \right |_P &= 
	\lim_{n \rightarrow \infty} \frac{1}{n} \sum_{i=1}^k d \left ( \pi_i(R_n(\gamma)) o_i, \left ( \left (\overline{o_i  \left (R_\infty(\gamma) \right )_i} \right )_{\flr{m_1 (\mu_i) n}} \right ) \right ) \\ &
	= \lim_{n \rightarrow \infty} \frac{1}{n} \sum_{i=1}^k |\mm{\imesi{i}}| \\
	&= 0.
	\end{align*}
	Hence, since $| \cdot |_P$ is subadditive, the less general criterion from Corollary \ref{cor:kaimanovichapproxthmtop} may be used instead of Theorem \ref{thm:kaimanovichapproxthmtop} to conclude that $(B, \nu)$ is the \PF boundary in the transitive action case.
\end{remark}

\section{Random walks and gauges on $\Vmma$}
We will now discuss the tree representation theory of Baumgartner and Willis \cite{baumgartner04} for totally disconnected, locally compact groups in the context of random walks. For convenience of the reader, we refer sometimes to the results from the previous sections, but because the action is on a single tree, the arguments of Cartwright, Kaimanovich and Woess \cite{cartwright94} are sufficiently general. We have used Horodam \cite{horodam15} in writing the preliminary part of this section. 

Suppose that $G$ is a totally disconnected, locally compact group. Let $\alpha$ be an automorphism of $G$. Let $V$ be a compact open subgroup of $G$ which is tidy for $\alpha$. Suppose that the order of $\alpha$ is infinite. Let $V_{-}, V_{--}, V_{+}, V_{++}$, $V_0$ and the scale function $s$ be as in the preliminary section.
	
Consider the semi-direct product $\Vmma $. Identify $\Vmm$ with the subgroup $\Vmm \times \{ e \}$ and $\Vm$ with the subgroup $\Vm \times \{ e \}$. For each $v$ in $\Vmm$ and each integer $m$, let $(v,m)$ be the left coset $v \alpha^m(\Vm)$ of $\Vm$ in $\Vmma$. Let $\Tdp$ be the directed graph whose vertices are the left cosets of $V_-$ in $\Vmma$ with an edge from $(v, m)$ to $(w, n)$ if and only if $n = m + 1$ and $w$ is an element of $v \alpha^m (\Vm).$

\begin{lemma}
	Every vertex of $\Tdp$ has degree $s(\alpha^{-1}) + 1$, with one inward facing edge at each vertex and $s(\alpha^{-1})$ outward facing edges at each vertex.
\end{lemma}
\begin{proof}
	Suppose $(v,n)$ is any vertex of the tree $\Tdp$. Let $(w_1, m_2)$ and $(w_1, m_2)$ be any vertices such that there are edges from them to $(w,n)$. Then, $m := m_1 = m_2 = n-1$, and $v$ must be in the intersection of $w_1 \alpha^m (V_-)$ and $w_2 \alpha^m (V_-)$. As the left cosets of $\Vm$ in $\Vmma$ must be equal or disjoint,  
	\[ w_1 \alpha^m (V_-) = w_2 \alpha^m (V_-), \]
	i.e. $(w_1, m_1)$ and $(w_2, m_2)$ are equal. Each vertex of $\Tdp$ therefore has only one inward facing edge.
	
	For the outward facing edges, suppose that $(w_2,\alpha^{n_2})$ and $(w_2,\alpha^{n_2})$ are elements of $\Vmm$ which represent the same left coset of $V_{-}$. Then, $n := n_2 = n_2$, and $(w_2,n) = (w_2,n)$, i.e. $w_2^{-1} w $ is in $\alpha^n(V_-)$. 
	Suppose that there is an edge from a vertex $(v,m)$ to $(w_2, n)$. Then, both $w_2$ and $w_2$ are in $v \alpha^{n-1} (V_-)$, hence $w_2^{-1} w_2$ is in $v \alpha^{n-1} (V_-)$. Therefore, there are $[\alpha^{n-1} (V_-) : \alpha^n (V)]$ choices of $(w,n)$ which result in distinct edges. Since $s(\alpha) = [\alpha(V_+) : V_+]$,  
	\[ s(\alpha^{-1}) = [\alpha^{-1}(V_-) : V_-] = [\alpha^{m}(V_-) : \alpha^{m+1}(V_-) ] \]
	each vertex has $s(\alpha^{-1})$ outward facing edges and every vertex of $\Tdp$ has degree $s(\alpha^{-1}) + 1$.
\end{proof}
The double path $P = \{ (e, n) \}_{n \in \I}$ is infinite as the order of $\alpha$ is infinite. Therefore, $\Tdp$ is infinite. 	Let $-\omega$ be the end of $\Tdp$ corresponding to infinite path $\{ (e, n) \}_{n \in \NN}$, and let $ \omega$ be the end corresponding to the infinite path $\{ (e, -n) \}_{n \in \NN}$.
If $(v,m)$ and $(w,n)$ are vertices of $\Tdp$, then the paths $\{(v,m+k) \}_{k \in \NN}$ and $\{(w,n+k) \}_{k \in \NN}$ both eventually ascend to an element of the path $P$, which is connected, so $\Tdp$ is connected. 
There are no cycles in $\Tdp$ because any cycle without backtracking must be a directed path, but the power of $\alpha$ increases strictly along directed edges, and there is only one edge directed into each vertex.

Because a vertex $v$ is a parent of a vertex $w$ if and only if $v = \left ( \overline{w -\omega} \right )_1$, there is no loss of information regarding $\Tdp$ as an undirected graph. We do so from now on. Then, $\Tdp$ is a homogeneous tree of degree $s(\alpha^{-1}) + 1$. We treat $V_-$ as a distinguished vertex, which we gave label $o$ in previous sections.

Suppose that the scale of $\alpha$ is not one, so that the degree of each vertex in $\Tdp$ is at least $3$. It is easy to verify that the left action of $\Vmma$ is a bijection on the vertex set of $\Tdp$, and it preserves adjacency of vertices. The element $(e, 1)$ in $\Vmma$ acts by translation by one edge on $P$ away from the distinguished end $\omega$ towards $-\omega$. Since every path descends from $\omega$, it is fixed by the action of $\Vmma$. The action is transitive on the other ends.

As before, denote by $\Aut \Tdp$ the group of all automorphisms of $\Tdp$ with the compact-open topology, by $\Aff \Tdp$ the closed subgroup of automorphisms which fix $\omega$ and by $\Hor \Tdp$ the subgroup of elliptic elements. Let $\pi$ be map from $\Vmma$ to $\Aut \Tdp$, which is the representation of the action of $\Vmma$ on $\Tdp$. Then:
\begin{enumerate}[(i)]
	\item $\pi$ is continuous, 
	\item the image of $\pi$ is a closed subgroup of $\Aff \Tdp$.
	\item the kernel of $\pi$ is the largest compact, normal, $\alpha$-stable subgroup of $V_{--}$ and
	\item the image of $V_{--}$ under $\pi$ is contained in $\Hor \Tdp$.
\end{enumerate}

Let $\eta = \phi \circ \pi$, so that if $(v,n)$ is in $\Vmma$, then $\eta(v,n) = n$. Let $\rho$ be the quotient map from $\Vmma$ to $\Vmma / \ker \pi$. The representation of $\Vmma / \ker \pi$ in $\Aff \Tdp$ is faithful. Any closed subgroup of the quotient group $\Vmma / \ker \pi$ isomorphic to a closed subgroup of $\Aff \Tdp$. 

The group $\Aff \Tdp$ is amenable. The quotient $\Vmma / \ker \pi$ is amenable, because it is isomorphic to a closed subgroup of $\Aff \Tdp$. The kernel of the map $\pi$ is amenable as it is compact, hence $\Vmma$ is amenable. 

According to Kaimanovich \cite{Kaimanovich2002}, if $G$ is any locally compact group with an amenable, closed, normal subgroup $H$ and $q$ is the quotient homomorphism from $G$ to $G / H$, then for any Borel probability measure $\mu'$ on $G / H$, there is a Borel probability measure $\mu$ on $G$, such that $\pf{q}{\mu} = \mu'$ and so that the \PF boundaries of $\Gmu$ and $(G / H, \mu')$ are isomorphic.

In particular, if $\mu'$ is any Borel probability measure $\mu'$ on  $\Vmma / \ker \pi$. Then there is a Borel probability measure $\mu$ on $\Vmma$, such that $\pf{\pi}{\mu} = \mu'$ and so that the \PF boundaries of $(\Vmma, \mu)$ and $(\Vmma / \ker \pi, \mu')$ are isomorphic.

We will now describe the \PF boundary of $(\Vmma, \mu)$ for the case where $\mu$ is an aperiodic, spread-out Borel probability measure on $\Vmma$, such that $\sgr \mu$ is a closed subgroup of $\Vmma$. We show that the \PF boundaries of $(\Vmma, \mu)$ and $(\Vmma / \ker \pi, \pf{\pi}{\mu} )$ are isomorphic in this case by adapting arguments from the single tree case.

For each element $(v,n)$ in $\Vmma$, let \[ |(v,n)|_{\Vmma} = |\pi(v,n)|_{\Tdp} = d(V_-, v \alpha^n (V_-)). \]
\begin{lemma}
	The map $|\cdot|_{\Vmma}$ is a subadditive gauge function, and $\mu$ has finite first moment with respect to $|\cdot|_{\Vmma}$ if and only if the pushforward measure $\pf{\pi}{\mu}$ has finite first moment with respect to   $|\cdot|_{\Tdp}$.
\end{lemma}
\begin{proof}
	Let $(v_1, n_1)$ and  $(v_2, n_2)$ be elements of $\Vmma$. Then,
	\begin{align*}
	|(v_1, n_1)(v_2, n_2)|_{\Vmma}  &= d(V_-, (v_1, n_1)(v_2, n_2) \cdot V_- ) \\
	&= d((v_1, n_1)^{-1} \cdot V_-, (v_2, n_2) \cdot V_- ) \\
	&\leq d((v_1, n_1)^{-1} \cdot V_-, V_-) + d(V_-, (v_2, n_2) V_-) \\
	&= |(v_1, n_1)|_{\Vmma}  + |(v_2, n_2)|_{\Vmma}.
	\end{align*}
	So $|\cdot|_{\Vmma}$ is a subadditive gauge function. Since $|\cdot|_{\Vmma} = |\pi(\cdot)|_{\Tdp}$, 
	\begin{align*}
	\int_{\Aff \Tdp} |x| \, d \pf{\pi}{\mu}(x) &= 
	\int_{\Aff \Tdp} |x| \, d \pf{\pi}{\mu}(x) \\ 
	&= \int_{\Vmma}  |x|_{\Vmma} \, d \mu(x)
	\end{align*}
	which gives the second part of the lemma. 
\end{proof}

Suppose that $\Gamma$ is a closed subgroup of $\Vmma$. Then, $\rho (\Gamma)$ is closed and $\pi(\Gamma)$ is isomorphic to $\pi(\Gamma)$. Hence, $\pi(\Gamma)$ is closed. We say that $\Gamma$ is \emph{non-exceptional} if it has non-exceptional image under $\pi$. Since the action of $\Gamma$ is transitive on $\partial^* \Tdp$, the closed subgroup $\Gamma$ is non-exceptional provided that $\eta(\Gamma)$ is isomorphic to an infinite cyclic group. 

We restate Theorem \ref{thm:roeafft} and Theorem \ref{theo:conv} in terms of $\Vmma$ and $\eta$:
\begin{theorem}
	\label{thm:vmmaroeafft}
	Let $R_m = (v^{(m)}, r^{(m)})$ be the right random walk associated with $(\Vmma, \mu)$. Suppose $\mu$ has finite first moment with respect to $|\cdot|_{\Vmma}$. Then, 
	\[ \lim_{m \rightarrow \infty} d \left ( R_m V_-,R_m V_- \right ) = 0 \]
	almost surely, and
	\[ \lim_{n \rightarrow \infty} \frac{1}{n} |  R_m |_{\Vmma} = \left | \mm{\pf{\eta}{\mu}} \right |  \]
	almost surely and in $L_1$.
\end{theorem}

\begin{theorem}
	Let $\mu$ be a Borel probability measure such that the closure of the semigroup generated by $\mu$ is a closed, non-exceptional subgroup $\Gamma$ in $\Vmma$.  Let $\left \{ R_m \right \}_{m=1}^\infty $ be the right random walk associated with $(\Vmma, \mu)$. Let $V_-$ be a distinguished vertex of $\Tdp$, and let $\omega$ be the fixed end. 
	\begin{enumerate}[(i)]
		\item If $m_1 (\pf{\eta}{\mu})$ is finite and the mean of $\pf{\eta}{\mu}$ is negative, then $R_m V_-$ converges almost surely to $\omega$.
		\item If $m_1 (\mu)$, with respect to $|(v,n)|_{\Vmma}$, is finite and the mean of $\pf{\eta}{\mu}$ is positive, then $R_m V_-$ converges to a random end in $\partial^* \Tdp$ almost surely.
		\item If $m_1 (\mu)$, with respect to $|(v,n)|_{\Vmma}$, is finite and the mean of $\pf{\eta}{\mu}$ is zero, and
		\[ \EV \left ( \left | V_- \wedge R_m^{-1} V_- \right | q^{|V_- \wedge R_m V_-|} \right ) \]
		is finite then $R_m V_-$ converges to $\omega$ almost surely.
	\end{enumerate}
\end{theorem}
For the remainder, suppose that $\mu$ is a spread-out Borel probability measure on $\Vmma$. Suppose $\mu$ has finite first moment with respect to   $|\cdot|_{\Vmma}$.  Suppose that $\mu$ is non-exceptional.
\begin{proposition}
	\label{prop:vmmagauges}
	Let $\mu$ be a Borel probability measure such that the closure of the semigroup generated by $\mu$ is a closed, non-exceptional subgroup $\Gamma$ in $\Vmma$.  Suppose that $\mu$ is has finite first moment with respect to $| \cdot |_{\Vmma} $. Let $\Gamma$ be the image of $\Vmma$ under $\pi$.  Let $R_m$ be the right random walk associated with $(\Vmma, \mu)$. Let $V_-$ be a distinguished vertex of $\Tdp$, and let $\omega$ be the fixed end. Let $u$ be an element of $B$ and let $n$ be a natural number. Let
	\[ \mathcal{B}^{(n)} (u) = \left \{ \left \{ g \in \Vmma : d \left ( \left (  \overline{V_- u} \right )_{\flr{n m_1(\pf{\eta}{\mu})}}, g V_- \right ) \leq j \right \} \right \}_{j=1}^\infty. \]
	for each natural number $j$. Then, $\{ \mathcal{B}^{(r)} (u) \}_{r=1}^\infty$ is a uniformly temperate sequence of gauges on $\Gamma$, and
	\[ \left | g \right |_{\mathcal{B}^{(n)} (u)} = d \left ( \left (  \overline{V_- u} \right )_{\flr{n m_1(\pf{\eta}{\mu})}}, g V_- \right ) \]
	is a gauge function on $\Vmma$. 
\end{proposition}
\begin{proof}
	The argument is similar to Proposition \ref{prop:sequniformlytemperategauges}. If $g$ is in $\Vmma$, then the distance from $g V_-$ to any vertex is always finite. It follows that
	\[ \Vmma = \bigcup_{j=1}^\infty  \left \{ g \in \Vmma : d \left ( \left (  \overline{V_- u} \right )_{\flr{n m_1(\pf{\eta}{\mu})}}, g V_- \right ) \leq j \right \}, \]
	i.e. $\mathcal{B}^{(n)} (u)$ eventually exhausts $\Vmma$.
\end{proof}

Let $R_\infty$ be the map from $(\Vmma)^\NN$ to $\partial T$ given by
\begin{align*}
R_\infty \left ( \gamma \right ) &= \begin{cases} \omega  & \textrm{if} \ m_1(\pf{\eta}{\mu}) \leq 0, \\ \lim_{n \rightarrow \infty}  R_m(\gamma)  V_- & \textrm{otherwise.} \end{cases} 
\end{align*}
Let $\nu = \pf{R_\infty}{\mu}$ be the hitting measure. Then $\nu$ is  $\mu$-stationary,  because
\begin{align*}
\int_{\Vmma} \int_B f(xb) \, d \nu(b) \, d \mu(x) &= \int_{\Vmma} \int_{(\Vmma)^\NN} f(x \, R_\infty(\gamma) ) \, d \pathmeasure(\gamma) \, d \mu(x) \\
&= \int_{(\Vmma)^\NN} f( R_\infty(T \gamma) ) \, d \pathmeasure( \gamma)  \\
&= \int_{(\Vmma)^\NN} f( R_\infty(\gamma) ) \, d \pathmeasure( \gamma)  \\
&= \int_B f(b) \, d \nu(b)
\end{align*}
where $T$ is the left shift map and $f$ is any function in $C(B)$. The $\Gmu$ space $(B, \nu)$ is a $\mu$-boundary because $R_\infty$ is shift invariant in the sense of Kaimanovich \cite{kaimanovich00} or  Brofferio and Schapira \cite{brofferio2011poisson} Proposition 2.1. It follows that $(\partial \Tdp, \nu)$ is a $\mu$--boundary of $(\Vmma, \mu)$. 
\begin{theorem}
	
	Let $\mu$ be a Borel probability measure such that the closure of the semigroup generated by $\mu$ is a closed, non-exceptional subgroup $\Gamma$ in $\Vmma$.  Suppose that $\mu$ is aperiodic, spread-out and has finite first moment with respect to $| \cdot |_{\Vmma} $. Then, $(\partial T, \nu)$ is the \PF boundary of $(\Vmma, \mu)$. 
\end{theorem}
\begin{proof}
	The argument is similar to Theorem \ref{thm:poisonnboundaryofproduct}. Suppose that $j$ is a natural number less than or equal to $k$. Let $u$ be in $\partial \Tdp$, let $n$ be a natural number and let the gauge from Proposition \ref{prop:vmmagauges}, 
	\[ \mathcal{B}^{(n)} (u) = \left \{ \left \{ (v,r) \in \Vmma : d \left ( \left (  \overline{V_- u} \right )_{\flr{n m_1(\pf{\eta}{\mu})}}, (v,r)V_- \right ) \leq j \right \} \right \}_{j=1}^\infty. \]
	Theorem 	\ref{thm:vmmaroeafft} gives that
	\[ \lim_{m \rightarrow \infty} d \left ( R_m V_-,(v^{(m+1)}, R_m V_- \right ) = 0 \]
	almost surely, and
	\[ \lim_{m \rightarrow \infty} \frac{1}{m} |  R_m  |_{\Vmma} = \left | \mm{\pf{\eta}{\mu}} \right |  \]
	almost surely and in $L_1$. In particular, the conditions of \ref{lem:regularity} are satisfied, that is, $R_m  V_-$ is almost surely a regular sequence of vertices in $\Tdp$, with rate of escape $|\mm{\pf{\eta}{\mu}}|$. Thus, for almost every path in the random walk, there is an end $u$ in $ \partial \Tdp$, such that
	\[ \lim_{n \rightarrow \infty}  d \left (  R_m V_-, (\overline{V_- u (\psi)})_{\flr{m_1 (\pf{\eta}{\mu}}} \right ) = |m_1 (\pf{\eta}{\mu})|.  \]
	If  $|m_1 (\pf{\eta}{\mu})| \neq 0$, then $R_m  V_- \rightarrow u (\psi)$, so $u$ is unique. By Theorem \ref{theo:conv}, $u = \omega$ if $m_1 (\pf{\eta}{\mu}) < 0$ and $u$ in $\partial^* \Td{p}$ if $m_1 (\pf{\eta}{\mu}) > 0$. If $|m_1 (\pf{\eta}{\mu})| = 0$, then $u_j (\psi)$ is arbitrary and the statement is just that
	\[ \lim_{n \rightarrow \infty}  d \left ( R_m V_-, V_- \right ) = 0.  \]
	Hence, for almost all paths $\gamma$, 
	\begin{align*}  
	\lim_{n \rightarrow \infty} \frac{1}{n} \left | R_m (\gamma)  \right |_{\mathcal{B}^{(n)} (R_\infty (\gamma))}	
	&= 
	\lim_{n \rightarrow \infty} \frac{1}{n} d \left ( R_m  V_- , \left (  \overline{V_- R_\infty (\gamma)} \right )_{\flr{n m_1(\pf{\eta}{\mu})}} \right ) \\ 
	&= 0.
	\end{align*}
	By assumption, $\mu$ has finite first moment with respect to the subadditive gauge associated with $|\cdot|_{\Vmma}$. Hence, by Kaimanovich's ray criterion for topological groups, Theorem \ref{thm:kaimanovichapproxthmtop}, $(\partial \Tdp, \nu)$ is the \PF boundary.
\end{proof}
\begin{corollary}
	Let $\mu$ be a Borel probability measure such that the closure of the semigroup generated by $\mu$ is a closed, non-exceptional subgroup $\Gamma$ in $\Vmma$.  Suppose that $\mu$ is aperiodic, spread-out and has finite first moment with respect to $| \cdot |_{\Vmma} $. Then, 
	\[ m_1(\pf{\eta}{\mu}) \leq 0 \]
	if and only if the \PF boundary of $(\Vmma, \mu)$ is trivial.
\end{corollary}
\begin{proof}
	The argument is the same as Proposition \ref{prop:uncountablesupportpositivedrift}. If $m_1(\pf{\eta}{\mu}) \leq 0$, then it follows from the definition of $\nu$ that the boundary is a single point.	If $ m_1(\pf{\eta}{\mu}) > 0 $, then $\nu$ is not a point measure.  
\end{proof}

\bibliographystyle{plain}

\end{document}